\documentclass{amsart}
\usepackage{color}
\usepackage[hyperfootnotes=false]{hyperref}
\usepackage{mathtools}
\usepackage{amsmath}
\usepackage{amssymb}
\usepackage{amsthm}
\usepackage{centernot}

\usepackage{amsrefs}

\newtheorem{theorem}{Theorem}[section]
\newtheorem{lemma}[theorem]{Lemma}

\newtheorem{proposition}[theorem]{Proposition}

\newtheorem{corollary}[theorem]{Corollary}

\theoremstyle{definition}
\newtheorem{definition}[theorem]{Definition}
\newtheorem{example}[theorem]{Example}
\newtheorem{algorithm}[theorem]{Algorithm}

\theoremstyle{remark}
\newtheorem{remark}[theorem]{Remark}

\DeclareMathOperator{\lf}{lf}
\DeclareMathOperator{\Syz}{Syz}
 
\allowdisplaybreaks

\begin{document}
\title{Macaulay bases of modules}

\author{Sujit Rao}
\address{CSAIL, MIT, Cambridge, MA 02139}
\email{sujit@mit.edu}

\date{\today}

\begin{abstract}
We define Macaulay bases of modules, which are a common generalization of Groebner bases and Macaulay $H$-bases to suitably graded modules over a commutative graded $\mathbf{k}$-algebra, where the index sets of the two gradings may differ. This includes Groebner bases of modules as a special case, in contrast to previous work on Macaulay bases of modules. We show that the standard results on Groebner bases and Macaulay $H$-bases generalize in fields of arbitrary characteristic to Macaulay bases, including the reduction algorithm and Buchberger's criterion and algorithm. A key result is that Macaulay bases, in contrast to Groebner bases, respect symmetries when there is a group $G$ acting homogeneously on a graded module, in which case the reduction algorithm is $G$-equivariant and the $\mathbf{k}$-span of a Macaulay basis is $G$-invariant. We also show that some of the standard applications of Groebner bases can be generalized to Macaulay bases, including elimination and computation of syzygy modules, which require the generalization to modules that was not present in previous work.
\end{abstract}

\maketitle
\tableofcontents

\section{Introduction}
Many computational tasks involving a module $M$ (or an ideal) over a finitely generated commutative ring $R$ can be solved by first computing a \textit{Groebner basis}, which is a special generating set of $M$ parameterized by a \textit{term order}. Groebner bases were first defined by Buchberger, who also gave an algorithm for constructing them. In the special case when $R = \mathbf{k}[x]$ is a polynomial ring in one variable over a field $\mathbf{k}$, Buchberger's algorithm is the same as a combination of Euclid's algorithm for greatest common divisor and polynomial long division. In multiple variables there is no longer a canonical ordering on monomials, so Groebner bases require introducing a term order so that long division can be generalized to a multivariate reduction algorithm. When the polynomials are linear, Buchberger's algorithm is the same as Gaussian elimination.
By now Groebner bases have become well known and several further applications have been developed, including \textit{elimination} of variables from a system of polynomials, and computing \textit{syzygy modules}. We use \cite{cocoa1} as our primary reference on Groebner bases, although we will reprove many of the relevant results from scratch in our more general setting.

While very useful in many cases, it is sometimes undesirable to choose a full term order, which forces us to distinguish variables in a polynomial ring and break symmetries. For example, we often encounter a module with an action by a group $G$, such as $S_{n}$ when an ideal is generated by symmetric polynomials, but computing a Groebner basis will almost always destroy the symmetry. One solution is to compute generators for the ring of polynomials invariant under $G$ and to re-express a presentation of the module in terms of these generators. However, the generators may not be algebraically independent, in which case there is not a canonical re-expression of the module. Additionally, computing generators for the invariant ring can be expensive unless the action is well-studied beforehand (such as for $S_{n}$) or the result can be pre-computed.

In this paper, we introduce \textit{Macaulay bases} for rings and modules graded by suitable monoids, which generalize Groebner bases for the monoids $\mathbb{N}^{d}$ and $(\mathbb{N}^{d})^{n}$ when $R$ is a polynomial ring in $d$ variables and $M \leq R^{\oplus n}$ is a submodule of a free module, generalize Macaulay $H$-bases of polynomial ideals (\cites{ms2003, sauer2001, cocoa2}) for the monoid $\mathbb{N}$, and generalize the Macaulay bases defined in \cite{cocoa2}*{Tutorial 48} when a ring is graded by a different monoid from the module on which it acts. In Sections~\ref{sec:prelim} and \ref{sec:reduction} we give generalizations of standard theorems about Groebner and $H$-bases, including analogues of the division algorithm and Buchberger's algorithm first proposed in \cite{sauer2001}. Our results provide a common generalization of those for Groebner and Macaulay $H$-bases, and extend the results on Macaulay $H$-bases in \cite{sauer2001} to fields of nonzero characteristic and to Macaulay bases of modules. In Section~\ref{sec:groupact} we give conditions under which a Macaulay basis respects the symmetry of a module with an action of a finite group. In Section~\ref{sec:applications} we give generalizations of the results used in several applications of Groebner bases and Macaulay $H$-bases, including elimination of variables, computation of Hilbert functions, finding generators of syzygy modules and free resolutions, and computation of homogenizations and projective closures.

\subsection{Previous work on Macaulay bases}
The definitions and results presented here generalize all of the previous ones known to the author. In particular, when $R$ is a polynomial ring graded by $\mathbb{N}^{d}$ and $M$ is a free $R$-module graded by $(\mathbb{N}^{d})^{n}$ we recover usual Groebner bases of modules; when $R$ is a polynomial ring graded by $\mathbb{N}$ using total degree and $M$ is a rank-1 free $R$-module we recover Macaulay $H$-bases of ideals as studied in \cites{ms2003, sauer2001, cocoa2}; and when the indexing sets of the gradings for $R$ and $M$ are the same we recover the Macaulay bases discussed on page-43 of \cite{cocoa2} (note that for Groebner bases of modules the indexing sets differ).

We additionally generalize the reduction algorithm of \cite{sauer2001}, which requires a family of inner products, to fields of nonzero characteristic where there may not be an inner product. We also provide further generalizations of many of the standard results on Groebner bases, including uniqueness of reduced bases, and the applications for elimination, Hilbert functions, syzygies, and homogenizations discussed in Section~\ref{sec:applications}. Moreover, in Section~\ref{sec:groupact} we state and prove a precise result that Macaulay bases can preserve group actions when the action preserves the gradings of $R$ and $M$, which was stated but not formalized or proved in \cite{sauer2001}. Of course, there are still several algorithmic issues which need to be solved before Macaulay bases can be used for practical computations and are discusses in Section~\ref{sec:issues}. In characteristic zero some work on the reduction algorithm was presented in \cite{sauer-red}.

\section{Preliminaries}

\label{sec:prelim}
Throughout, let $\mathbf{k}$ be a field $S = \mathbf{k}[x_{1}, \dots, x_{d}]$ be a polynomial ring in $d$ variables over $\mathbf{k}$. We first introduce some notation and then generalize the definition of Groebner bases.

\begin{definition}
A commutative monoid $(A, +, 0)$ is \textbf{totally ordered} if there is a total order $\leq$ on $A$ such that $0 < a$ for all $a \neq 0$ and $a \leq b$ implies $a + c \leq b + c$ for all $a, b, c \in A$.
\end{definition}

\begin{definition}
Let $A$ be a totally ordered commutative monoid and $R$ be a commutative ring. An \textbf{$A$-grading} of $R$ is a decomposition $R = \bigoplus _{a \in A} R_{a}$ into abelian groups where $R_{a}R_{b} \subseteq R_{a + b}$ for all $a, b \in A$.
If $M$ is an $R$-module and $B$ is a totally ordered set on which $A$ acts monotonically, then a \textbf{$B$-grading} of $M$ is a decomposition $M = \bigoplus _{b \in B} M_{b}$ into abelian groups where $R_{a}M_{b} \subseteq M_{a \cdot b}$.
The \textbf{degree} and \textbf{leading form} of $m \in M$ are
\begin{align*}
\deg m &\coloneqq \max \{b \in B : m_{b} \neq 0\} \\
\lf m &\coloneqq m_{\deg m}
\end{align*}
respectively, where $m = \sum _{b \in B} m_{b}$ is the decomposition of $m$ according to the $B$-grading.
An element $m \in M$ is \textbf{homogeneous} of degree $b$ if $m \in M_{b}$. These extend to the case when $m \in R$ by considering $R$ as an $A$-graded free module over itself.
\end{definition}

When otherwise unstated, we will assume that $R$ is an $A$-graded $\mathbf{k}$-algebra for some field $\mathbf{k}$, and that $M \leq N$ is an $R$-submodule of some $B$-graded $R$-module $N$. This includes the case when $M$ is an ideal, as we can take $N = R$ and $B = A$. In many cases, and in particular for describing certain reduction algorithms, we will also need to assume that $\operatorname{char} \mathbf{k} = 0$, although we will explicitly state when a result is independent of $\operatorname{char} \mathbf{k}$. Additionally, we will assume that $A$ and $B$ are \textbf{cancellative}, that is, $a + c = b + c$ implies $a = b$ for all $a, b, c$.

To guarantee that any computations terminate we will typically need to assume that $R$ is noetherian, which by the Hilbert basis theorem is true whenever $R$ is a polynomial ring over $\mathbf{k}$ in a finite number of variables. For the same reason we will also need to assume that $A$ and $B$ are both finitely generated and that their orderings are well-orderings, allowing us to induct over the orders.

\begin{definition}
A partial order $\leq$ on a commutative monoid $A$ is a \textbf{monoid ordering} if $a \leq b$ implies $a + c \leq b + c$ for all $a, b, c \in A$ and $0 \leq a$ for all $a \in A$.
\end{definition}

\begin{definition} \label{def:macbasis}
Suppose $A$ and $B$ have given total monoid orderings. Then $m_{1}, \dots, m_{k} \in M$ is a \textbf{Macaulay basis} with respect to the grading if we have equality of the submodules
\[ (\lf(p) : p \in M) = (\lf(f_{1}), \dots, \lf(f_{k})). \]
\end{definition}

\begin{example} \label{ex:mongrad}
Any monomial order $\preceq$ produces a corresponding total monoid order over the monoid $(\mathbb{N}^{d}, +, 0)$. There is also a canonical $\mathbb{N}^{d}$-grading of $S$ where $S_{a_{1}, \dots, a_{d}} = \langle x_{1}^{a_{1}} \cdots x_{d}^{a_{d}} \rangle$. A Macaulay basis of an ideal $I \leq S$ with respect to this grading and ordering is the same as a Groebner basis with respect to the monomial order $\preceq$.
\end{example}

\begin{example} \label{ex:grobgrade}
The module $N^{\oplus n}$ is graded by the monoid $(\mathbb{N}^{d})^{n}$, and $\mathbb{N}^{d}$ acts on $(\mathbb{N}^{d})^{n}$ by addition on each component. If in addition we put a term order on $(\mathbb{N}^{d})^{n}$, a Macaulay basis of a submodule $M \leq (\mathbb{N}^{d})^{\oplus n}$ is the same as a Groebner basis with respect to the term order.
\end{example}

\begin{remark}
If $A$ is a totally ordered commutative monoid and is also cancellative, we can extend the order on $M$ to its Grothendieck group. In that case the existence of a total ordering shows that the Grothendieck group is torsion-free (as torsion implies the ordering will contain a cycle), and since it is also finitely generated it must be isomorphic to $\mathbb{Z}^{n}$ for some $n$ as an (unordered) group. In particular, since the total ordering is compatible with the monoid structure it must refine the Green partial order (defined by $a \leq b$ if $a + c = b$ for some $c$). This implies that $A$ is well-ordered and hence we can induct on its elements according the ordering. Moreover, the order type of $A$ is $\omega$, so in induction proofs we will sometimes write $a + 1$ to denote the element which covers $a$ in the total order.
\end{remark}

\begin{example} \label{ex:macbasis}
The \textbf{standard grading} of a polynomial ring is $S = \bigoplus _{k \geq 0} S_{k}$ where $S_{k}$ is the vector space of homogeneous polynomials of degree $k$. A Macaulay basis of an ideal with respect to the standard grading is known as a \textbf{Macaulay $H$-basis} (where $H$ stands for ``homogeneous'') \cite{sauer2001}. Macaulay $H$-bases were first investigated in the early 1900's by Macaulay, and were further developed and formalized by Sauer in \cite{sauer2001} in a modern style analogous to that of Groebner bases.

Definition \ref{def:macbasis} also gives two possible generalizations of Macaulay $H$-bases to submodules of free modules. In the first we grade $S^{\oplus n}$ by $\mathbb{N}$, where $(S^{\oplus n})_{k} = (S_{k})^{\oplus n}$. Alternatively we can grade $S^{\oplus n}$ by $\mathbb{N}^{n}$, where
\[ (S^{\oplus n})_{(k_{1}, \dots, k_{n})} = \bigoplus _{i = 1} ^{n} S_{k_{i}}. \]
\end{example}

\begin{example}
Suppose $R$ is an $A$-graded ring and $N = \bigoplus _{i = 1} ^{r} R(a_{i})$ is an $A$-graded module, where $R(a_{i})$ is the $A$-graded $R$-module with grading
\[ R(a_{i})_{a} = \begin{cases}
R_{a'} & \text{there exists $a' \in A$ such that $a' + a_{i} = a$} \\
0 & \text{otherwise}
\end{cases} \]
and the direct sum of two $A$-graded modules $M_{1}$ and $M_{2}$ has the grading $(M_{1} \oplus M_{2})_{a} = (M_{1})_{a} \oplus (M_{2})_{a}$. Letting $A$ act on itself by $a \cdot a' = a + a'$, we recover the definition of Macaulay bases in Tutorial 48 of \cite{cocoa2}. The definition of Macaulay bases given here allows the index set of the grading of $N$ to be different from $A$. One example is the second generalization of Macaulay bases to modules given in Example~\ref{ex:macbasis}. Another is the standard definition of Groebner bases of modules, which are not a special case of the Macaulay bases defined in \cite{cocoa2} but are a special case of the definition given here, as shown in Example~\ref{ex:grobgrade}.
\end{example}

\section{Reduction algorithm and characterizations of Macaulay bases}
\label{sec:reduction}
In this section we give a generalization of Sauer's reduction algorithm \cite{sauer2001} for Macaulay bases. We also state some equivalent characterizations of Macaulay bases, which are easy generalizations of standard characterizations of Macaulay bases and of Groebner bases found in \cites{cocoa1, cocoa2, ms2003}.

\begin{definition}
Given a vector $(m_{1}, \dots, m_{n}) \in M^{\oplus n}$, its \textbf{syzygy module} is
\[ \Syz(m_{1}, \dots, m_{n}) = \left\{(f_{1}, \dots, f_{n}) \in R^{\oplus n} \middle| 0 = \sum _{i = 1} ^{n} f_{i}m_{i}\right\}. \]
An element of the syzygy module is called a \textbf{syzygy}.
\end{definition}

The fact that graded components of $R$ and $M$ can have dimension greater than one over $\mathbf{k}$ necessitates the following definition in order to state some of the results on reduction and normal forms with respect to Macaulay bases, although in some cases we will be able to avoid it.

\begin{definition}
A \textbf{complemented vector space} is a vector space $V$ over $\mathbf{k}$ along with a map $(-)^{c}$ on subspaces of $V$ sending each subspace to a fixed complement, i.e.\ $W \oplus W^{c} = V$ for all subspaces $w \leq V$. A \textbf{graded complemented vector space} is a graded vector space $V$ where each graded component is a complemented vector space. When $V$ is a graded module over a graded $\mathbf{k}$-algebra, we consider it to be a graded vector space by treating each of its graded components as vector spaces.
\end{definition}

\begin{example}
Our primary example of complemented vector spaces will come from graded vector spaces over a field of characteristic zero with a \textbf{graded inner product}. Specifically, if $V = \bigoplus _{b \in B} V_{b}$ is a graded vector space, a graded inner product on $V$ is an inner product $(\cdot, \cdot)$ on $V$ such that $V_{b} \perp V_{b'}$ for $b \neq b'$. The complement of a subspace $W \leq V_{b}$ is then defined to be the orthogonal complement
\[ W^{c} \coloneqq \{v \in V_{b} \mid \text{$(v, w) = 0$ for all $w \in W$}\}. \]
If $U \leq V_{b}$ is another subspace, then we define $U \ominus W \coloneqq U \cap W^{c}$.
\end{example}

We now give the following definition of the reduction algorithm, which generalizes the standard reduction algorithm for Groebner bases and the $H$-basis reduction algorithm from \cite{sauer2001}. Note that the reduction algorithm of \cite{sauer2001} requires a choice of an inner product, which can only be done when $\operatorname{char} \mathbf{k} = 0$. A canonical choice of complements can still be made when $\operatorname{char} \mathbf{k} \neq 0$, in which case the following reduction algorithm still satisfies the usual properties.

\begin{definition}[Reduction algorithm]
\label{def:reduction}
Suppose $N$ has a given structure as a graded complemented vector space.
Given $X = \{m_{1}, \dots, m_{n}\} \subseteq M$, define the subspaces
\[ W_{b}(X) \coloneqq \operatorname{span}_{\mathbf{k}} \{ r(\lf m_{i}) : r \in R, \deg r m_{i} = b, 1 \leq i \leq n \} \leq N_{b} \]
for all $b \in B$. We define a binary relation $\rightarrow_{X} \subseteq M^{2}$ on $M$ as follows: $m \rightarrow_{X} m'$ if $m' = m - \sum _{j = 1} ^{k} r_{j}m_{j}$ where $m_{j} \in X$ for $1 \leq j \leq k$ and
\[ m_{b} = \sum _{j = 1} ^{k} r_{j}(\lf m_{j}) \in W_{b}(X) \]
with $b = \max \{b' \leq \deg m \mid m_{b'} \in W_{b'}(X) \}$. If furthermore $N$ has a given structure as a graded complemented vector space, then we define the relation $m \Rightarrow_{X} m'$ if $m' = m - \sum _{j = 1} ^{k} r_{j}m_{j}$ where $m_{j} \in X$ for $1 \leq j \leq X$, but weaken the second condition to only require that
\[ m_{b} - \sum _{j = 1} ^{k} r_{j}(\lf m_{j}) \in W_{b}(X)^{c} \leq N_{b} \]
with $b = \max \{b' \leq \deg m \mid m_{b'} \notin W_{b'}(X)^{c}\}$.
Otherwise $\rightarrow_{X}$ and $\Rightarrow_{X}$ contain no other elements in $M^{2}$.

The relations $\rightarrow_{X}^{*}$ and $\Rightarrow_{X}^{*}$ are defined to be the reflexive and transitive closures of $\rightarrow_{X}$ and $\Rightarrow_{X}^{*}$ respectively. We say that $m$ \textbf{reduces} to $m'$ with respect to $X$ if $m \rightarrow_{X}^{*} m'$, with context determining if we use $\rightarrow_{X}^{*}$ or $\Rightarrow_{X}^{*}$. If there is no $m'$ such that $m \rightarrow_{X} m'$ or $m \Rightarrow_{X}^{*} m'$, then $m$ is \textbf{reduced} with respect to $X$.
\end{definition}

\begin{example}
For Groebner bases over an arbitrary field, each graded component of $N$ has dimension one over $\mathbf{k}$. Thus $N$ has a unique structure as a graded complemented vector space. In this case the definition of the reduction algorithm agrees with the usual Groebner basis reduction algorithm.
\end{example}

\begin{remark}
For doing concrete computations, we will typically have a canonical graded inner product on $N$. To execute the reduction algorithm, it is then necessary to compute bases of the spaces $W_{b}(X)$ and to perform orthogonal projections onto it. The task of computing a basis was studied for the special case of $H$-bases in \cite{sauer-red}. However, since $\dim_{\mathbf{k}} W_{b}(X)$ and $\dim_{\mathbf{k}} M_{b}$ can be very large it is likely impractical to store full bases for all $W_{b}(X)$, and ideally we would be able to compute orthogonal projections onto $W_{b}(X)$ without needing to compute or store an explicit basis.
\end{remark}

The following fact is standard for term rewriting systems and also necessary for our purposes.
\begin{lemma}
The relations $\rightarrow_{X}$ and $\Rightarrow_{X}$ are Noetherian, i.e.\ have no infinite chains.
\end{lemma}
\begin{proof}
For $\rightarrow_{X}$, we see that if $m \rightarrow_{X} m'$ then
\[ \max \{b' \leq \deg m \mid m_{b'} \in W_{b'}(X) \} < \max \{b' \leq \deg m' \mid m_{b'} \in W_{b'}(X) \}. \]
Since $B$ has no infinite descending chains, this implies that a reduction sequence must eventually terminate. Similarly, if $m \Rightarrow_{X}^{*} m'$ then
\[ \max \{b' \leq \deg m \mid m_{b'} \notin W_{b'}(X)^{c}\} < \max \{b' \leq \deg m' \mid m_{b'} \notin W_{b'}(X)^{c}\}. \qedhere \]
\end{proof}

The following proposition connects the two possible definitions of the reduction algorithm in Definition \ref{def:reduction}.

\begin{proposition}
\label{prop:reduction-equiv}
Suppose $N$ has the structure of a complemented vector space. Then $m \rightarrow_{X}^{*} 0$ if and only if $m \Rightarrow_{X}^{*} 0$.
\end{proposition}
\begin{proof}
Since the reduction relations are Noetherian, we may induct on the number of reduction steps. First suppose $m \rightarrow_{X} m'$ and let $b = \deg m$. Then $m_{b} \in W_{b}(X)$, so the projection of $m_{b}$ onto $W_{b}(X)^{c}$ is 0. Thus $m \Rightarrow_{X} m'$. A straightforward induction argument shows that if $m \rightarrow_{X}^{*} 0$ then $m \Rightarrow_{X}^{*} 0$.

Now suppose $m \Rightarrow_{X}^{*} 0$, and consider the first reduction step $m \Rightarrow_{X} m'$. Let $b = \deg m$. Then $m'_{0} = 0$ since $m \Rightarrow_{X}^{*} 0$. Thus $m_{b} \in W_{b}(X)$ and so $m \rightarrow_{X} m'$. A similar induction argument shows that $m \rightarrow_{X}^{*} 0$.
\end{proof}

We will say that an element $m \in M$ reduces to zero with respect to a finite set $X \subseteq M$ if $m \rightarrow_{X}^{*} 0$ or $m \Rightarrow_{X}^{*} 0$ depending on context, and by Proposition \ref{prop:reduction-equiv} we are justified in ignoring the difference between the two relations.

We are now ready to state and prove the following theorem, which gives several equivalent characterizations of Macaulay bases. This generalizes Theorem~2.4.1 of \cite{cocoa1} for Groebner bases of modules, the discussion starting with Section~4.2.B of \cite{cocoa2}, and Theorem~2.3 of \cite{ms2003} for Macaulay H-bases of ideals.

\begin{theorem}
\label{thm:conditions}
Let $X = \{m_{1}, \dots, m_{n}\} \subseteq M$. Then the following conditions are equivalent:
\begin{enumerate}
\item The set $X$ is a Macaulay basis of $M$. \label{lst:macbasis}
\item Every $m \in M$ reduces to zero with respect to $X$. \label{lst:redchar}
\item For every $m \in M$, there are $r_{1}, \dots, r_{n} \in R$ such that $m = \sum _{i = 1} ^{n} r_{i}m_{i}$ and $\deg r_{i}m_{i} \leq \deg m$ for all $1 \leq i \leq n$. \label{lst:expchar}
\item If $(s_{1}, \dots, s_{n}) \in \Syz(\lf m_{1}, \dots, \lf m_{n})$ then $\sum _{i = 1} ^{n} s_{i}m_{i}$ reduces to zero with respect to $X$. \label{lst:buchberger}
\end{enumerate}
\end{theorem}
\begin{proof}
(\ref{lst:macbasis}) $\implies$ (\ref{lst:redchar}) It suffices to show that if $m \in M$ is nonzero then we can always take a reduction step, since we can keep taking reduction steps until we reach 0. We have $\lf m \in (\lf m_{1}, \dots, \lf m_{n})$ since $m_{1}, \dots, m_{n}$ is a Macaulay basis, so $\lf m \in W_{\deg m}(m_{1}, \dots, m_{n})$ and thus we can take a reduction step.

(\ref{lst:redchar}) $\implies$ (\ref{lst:expchar}) Suppose $m \in M$, and consider the reduction algorithm as it reduces $m$ to 0. At each step it replaces $m$ with $m - \sum _{i = 1} ^{n} r_{i} m_{i}$ where $\deg r_{i}m_{i}$ is the current degree, so summing these expressions over all steps gives an expression $m - \sum _{i = 1} ^{n} r_{i}m_{i} = 0$ where $\deg r_{i}m_{i} \leq \deg m$.

(\ref{lst:redchar}) $\implies$ (\ref{lst:buchberger}) Clearly $\sum _{i = 1} ^{n} s_{i}m_{i} \in M$, so it reduces to zero.

(\ref{lst:expchar}) $\implies$ (\ref{lst:macbasis}) Suppose $m = \sum _{i = 1} ^{n} r_{i}m_{i}$ with $\deg r_{i} m_{i} \leq \deg m$. Put
\[ r_{i}' = \begin{cases}
\lf r_{i} & \text{if $\deg r_{i}m_{i} = \deg m$} \\
0 & \text{otherwise.}
\end{cases} \]
Then
\[ \lf m = \sum _{i = 1} ^{n} \lf r_{i}m_{i} = \sum _{i = 1} ^{n} r_{i}' \lf m_{i} \in (\lf m_{1}, \dots, \lf m_{n}). \]

(\ref{lst:buchberger}) $\implies$ (\ref{lst:expchar}) Let $m \in M$ with $m = \sum _{i = 1} ^{n} r_{i}m_{i}$. We use induction on $a = \max _{i} \deg r_{i}m_{i}$ to show that there are $s_{1}, \dots, s_{n}$ such that $m = \sum _{i = 1} ^{n} s_{i} m_{i}$ and $\deg s_{i}m_{i} \leq \deg m$. The base case $a = \deg m$ is clear. For the induction step, put
\[
r_{i}' = \begin{cases}
\lf r_{i} & \text{if $\deg r_{i}m_{i} = a$} \\
0 & \text{otherwise}.
\end{cases}
\]
Then $(r_{i}', \dots, r_{n}')$ is a syzygy of $(\lf m_{1}, \dots, \lf m_{n})$ since the degree $a$ homogeneous components of both $m$ and $\sum _{i = 1} ^{n} r_{i}m_{i}$ are 0. By assumption $\sum _{i = 1} ^{n} r_{i}'m_{i}$ reduces to 0, so we can find a representation $\sum _{i = 1} ^{n} r_{i}'m_{i} = \sum _{i = 1} ^{n} r_{i}''m_{i}$ with
\[ \deg r_{i}''m_{i} \leq \deg \sum _{i = 1} ^{n} r_{i}'m_{i} < \deg m \]
where the strict inequality holds because $(r_{1}', \dots, r_{i}')$ is a syzygy of $(\lf m_{1}, \dots, \lf m_{n})$. We can then apply the induction hypothesis to $m = \sum _{i = 1} ^{n} (r_{i} - r_{i}' + r_{i}'')m_{i}$ to get the desired representation.
\end{proof}

For Groebner bases and $H$-bases, it is known that an element of $m$ has a unique normal form with respect to the reduction algorithm \citelist{\cite{cocoa1}*{Proposition 2.2.5} \cite{ms2003}*{Lemma 6.2}}. When $N$ is a graded complemented vector space there also holds for Macaulay bases of modules, as shown in the following proposition.

\begin{proposition}
Suppose $N$ is a graded complemented vector space and let $X = \{m_{1}, \dots, m_{n}\} \subseteq M$ be a Macaulay basis of $M$. Then every $m \in N$ reduces to a unique normal form with respect to $X$, and the normal form is contained in
\[ \bigoplus _{b \in B} N_{b} \ominus W_{b}(X) \]
\end{proposition}
\begin{proof}
Suppose that $X$ is a Macaulay basis, and let $m \in N$. Since $\Rightarrow_{X}^{*}$ is Noetherian, there is at least one normal form $m'$ such that $m \Rightarrow_{X}^{*} m'$. Then $m - m' \in M$. If $m''$ is another normal form, then $m - m'' \in M$, so $m' - m'' \in M$. Since $X$ is a Macaulay basis $m' - m''$ reduces to 0, and since it cannot be reduced we must have $m' - m'' = 0$. From the definition of the reduction algorithm, we can easily see that
\[ m' \in \bigoplus _{b \in B} N_{b} \ominus W_{b}(X) \]
since otherwise we would be able to take another reduction step.
\end{proof}

The result on normal forms allows us to compute canonical $\mathbf{k}$-bases for the quotient $N/M$.

\begin{corollary} \label{cor:normform}
If $X \subseteq M$ is a Macaulay basis and $N$ is a complemented graded vector space, then
\[ N = M \oplus \bigoplus _{b \in B} N_{b} \ominus W_{b}(X). \]
\end{corollary}

In particular, this allows us to compute a basis for the quotient ring $R/I$ when we have a Macaulay basis of an ideal $I$. If $R$ is a polynomial ring and $I$ is zero-dimensional, then combining this with the M\"{o}ller-Stetter algorithm \cite{moeller-stetter} and any standard eigenvalue algorithm shows how Macaulay bases can be used to solve zero-dimensional polynomial systems. 

\begin{remark}
We have not characterized Macaulay bases in terms of the uniqueness of normal forms or confluence of the reduction relation, but only shown how a Macaulay basis allows us to compute unique normal forms. This is because the reduction relation in Definition~\ref{def:reduction} always reduces the terms of highest possible degree, making the reduction completely deterministic. If desired, standard results on term rewriting systems in general, such as the ones found in  \cite{cocoa1}, will give alternative characterizations of Macaulay bases in terms of confluence or the existence of unique normal forms for nondeterministic analogues of the reduction relation.
\end{remark}

\subsection{Refinements of gradings}
One can observe that a Groebner basis with respect to a degree-compatible monomial order is also a Macaulay $H$-basis. Intuitively this holds because the grading given by a degree-compatible monomial order ``refines'' the standard grading. We make precise the notion of refinement and generalize this observation. The notion of refinement will also be useful for generalizing Buchberger's algorithm to Macaulay bases.

\begin{definition}
Suppose $R$ has two given gradings by $A$ and $A'$. Then the grading by $A$ is a \textbf{refinement} (or \textbf{refines}) the grading by $A'$ if there is an order-preserving monoid homomorphism $f : A \to A'$ such that $R_{a'} = \bigoplus _{a \in f^{-1}(\{a'\})} R_{a}$.

For modules, assume that $R$ has a grading by $A$ which refines a grading by $A'$. If $B$ and $B'$ are just totally ordered sets, and $M$ has two gradings by $B$ and $B'$, then the definition of refinement is the same. When $B$ and $B'$ have appropriate actions by $A$ and $A'$ respectively, then we require the compatibility condition
\[ f(a) \cdot g(b) = g(a \cdot b) \]
for all $a \in A$ and $b \in B$, where $f : A \to A'$ and $g : B \to B'$ are the refinement maps. 
\end{definition}

\begin{proposition}
Suppose $m_{1}, \dots, m_{n}$ is a Macaulay basis of $M$ with respect to a grading by $B$ which refines a grading by $B'$. Then $m_{1}, \dots, m_{n}$ is also a Macaulay basis with respect to the grading by $B'$.
\end{proposition}
\begin{proof}
Let $m \in M$. Since $m_{1}, \dots, m_{n}$ is a Macaulay basis with respect to $B$, there are $r_{1}, \dots, r_{n} \in R$ such that $m = \sum _{i = 1} ^{n} r_{i}m_{i}$ and $\deg_{B}(m) \geq \deg_{B}(r_{i}m_{i})$ for all $i$. Let $f : A \to A'$ and $g : B \to B'$ be monotone homomorphisms as in the definition of refinement. If $m_{b}$ denotes the homogeneous component of an arbitrary $m \in M$ at $b \in B$ then the homogeneous component at $b' \in B'$ is $\sum _{b \in f^{-1}(\{b'\})} m_{b}$, so $\deg_{B'}(m) = g(\deg_{B}(m))$. Thus
\[
\deg_{B'}(m) = g(\deg_{B}(m)) \geq g(\deg_{B}(r_{i}m_{i})) = \deg_{B'}(r_{i}m_{i})
\]
for all $i$. By Theorem \ref{thm:conditions}, $m_{1}, \dots, m_{n}$ is a Macaulay basis with respect to the grading by $B'$.
\end{proof}

In order to generalize Buchberger's algorithm and Schreyer's theorem, we will need to define a special grading on syzygy modules of homogeneous elements. Our definition is in terms of refinements.

\begin{example}
If $A = \mathbb{N}^{d}$ has an order given by a degree-compatible monomial order and $S$ is a polynomial ring with the standard grading by $B = \mathbb{N}$, then $A$ refines $B$ via the homomorphism $f(a) = \sum _{i = 1} ^{d} a_{i}$. Monotonicity of $f$ is equivalent to the statement that the monomial order is degree-compatible.
\end{example}

\begin{definition}
\label{def:shiftgrade}
Suppose $b_{1}, \dots, b_{n} \in B$ is a finite sequence. Then we define a grading on $R^{\oplus n}$ that refines the grading by $A^{\oplus n}$ via the homomorphism $f : A^{\oplus n} \to B$ given by $f(a_{1}, \dots, a_{n}) = \sum _{i = 1} ^{n} a_{i}b_{i}$, which gives
\[ (R^{\oplus n})_{b} = \bigoplus _{c \in f^{-1}(\{b\})} (R^{\oplus n})_{c} \]
for all $b \in B$.
\end{definition}

\begin{lemma}
\label{lem:gradsyz}
Suppose $m_{1}, \dots, m_{n} \in M$ are homogeneous. Then $\Syz(m_{1}, \dots, m_{n})$ is a graded submodule of $R^{\oplus n}$ when equipped with grading defined in Definition~\ref{def:shiftgrade} associated to the sequence $\deg m_{1}, \dots, \deg m_{n}$.
\end{lemma}
\begin{proof}
Let $s = (s_{1}, \dots, s_{n}) \in \Syz(m_{1}, \dots, m_{n})$. Then $0 = \sum _{i = 1} ^{n} s_{i}m_{i}$. Equating homogeneous components on both sides gives
\[ 0 = \left( \sum _{i = 1} ^{n} s_{i}m_{i} \right)_{b} = \sum _{(a_{1}, \dots, a_{n}) \in f^{-1}(\{b\})} \sum _{i = 1} ^{n} (s_{i})_{a_{i}}m_{i} \]
so
\[
s_{b} = \sum _{(a_{1}, \dots, a_{n}) \in f^{-1}(\{b\})} ((s_{1})_{a_{1}}, \dots, (s_{n})_{a_{n}}) \in \Syz(m_{1}, \dots, m_{n})
\]
where $s_{b}$ is the degree-$b$ homogeneous component of $s$.
\end{proof}

Since $\lf m_{1}, \dots, \lf m_{n}$ are always homogeneous, the above proposition shows\[ \Syz(\lf m_{1}, \dots, \lf m_{n}) \leq R^{\oplus n} \]
is a graded submodule for any $m_{1}, \dots, m_{n} \in M$.

\subsection{Buchberger's algorithm}
We can weaken condition~\ref{lst:buchberger} of Theorem \ref{thm:conditions} slightly to obtain a generalization of Buchberger's criterion and algorithm, which was already generalized to $H$-bases in \cite{ms2003}*{Theorem 4.2}.

\begin{theorem}[Buchberger's criterion]
\label{thm:criterion}
Fix a generating set $s_{1}, \dots, s_{k}$ of the module $\Syz(\lf m_{1}, \dots, \lf m_{n})$ where every element is homogeneous with respect to the grading defined in Lemma~\ref{lem:gradsyz}. Then $X = \{m_{1}, \dots, m_{n}\}$ is a Macaulay basis of $M$ iff $\sum _{i = 1} ^{n} s_{ji} m_{i} \rightarrow_{X}^{*} 0$ for all $j$.
\end{theorem}
\begin{proof}
If $X$ is a Macaulay basis, then $\sum _{i = 1} ^{n} s_{ji}m_{i}$ is an element of $M$ and thus reduces to 0 by Theorem~\ref{thm:conditions}. Now suppose that $\sum _{i = 1} ^{n} s_{ji}m_{i} \rightarrow_{X}^{*} 0$ for all $j$. Let $s \in \Syz(\lf m_{1}, \dots, \lf m_{n})$. We will show that $\sum _{i = 1} ^{n} s_{i}m_{i} \rightarrow_{X}^{*} 0$, which by Theorem~\ref{thm:conditions} implies that $X$ is a Macaulay basis.

The reduction relation is easily seen to be $\mathbf{k}$-linear, so it suffices to show that $\sum _{i = 1} ^{n} (s_{b})_{i}m_{i} \rightarrow_{X}^{*} 0$ for all $b \in B$ and then sum over all $b$. Let $s_{b} = \sum _{j = 1} ^{k} r_{j}s_{j}$, where $r_{j}s_{j}$ is homogeneous of degree $b$. Since $\sum _{i = 1} ^{n} (s_{j})_{i}m_{i} \rightarrow_{X}^{*} 0$ we also have $r_{j}\sum _{i = 1} ^{n} (s_{j})_{i}m_{i} \rightarrow_{X}^{*} 0$. Summing over $j$ shows that $\sum _{i = 1} ^{n} (s_{b})_{i}m_{i} \rightarrow_{X}^{*} 0$.
\end{proof}

The following algorithm generalizes Buchberger's algorithm for Groebner bases of ideals and modules as well as Sauer's analogue in \cite{ms2003} for $H$-bases of ideals.

\begin{algorithm}[Buchberger's algorithm]
\label{alg:buchberger}
Begin with the set $X$ initialized to a provided generating set of $M$. 
\begin{enumerate}
\item Compute a homogeneous generating set $s_{1}, \dots, s_{k} \in \Syz(\lf m_{1}, \dots, \lf m_{n})$, where $X = \{m_{1}, \dots, m_{n}\}$.
\item For each $s_{j} = (s_{j1}, \dots, s_{jn})$, reduce $\sum _{i = 1} ^{n} s_{ji}m_{i}$ to a normal form $m'_{j}$.
\item Set $Y = \{m'_{j} : 1 \leq j \leq k, m'_{j} \neq 0\}$. If $Y = \emptyset$, then halt and output $X$. Otherwise, set $X = X \cup Y$ and go back to step 1.
\end{enumerate}
\end{algorithm}

\begin{theorem}
Assuming a black-box for computing syzygy modules, Algorithm~\ref{alg:buchberger} terminates in a finite amount of time, and its output is a Macaulay basis of $M$.
\end{theorem}
\begin{proof}
To see that the algorithm terminates, observe that the submodule $(\lf X)$ strictly increases at each step, since each $m'_{j} \in Y$ is reduced to a normal form and is nonzero. This gives a strictly increasing chain of submodules of $N$ for each iteration of the algorithm, and since $N$ is Noetherian the chain must eventually terminate. Correctness follows from Theorem~\ref{thm:criterion}, since the algorithm only halts when every reduction reaches 0.
\end{proof}

Explicitly bounding the running time of Algorithm~\ref{alg:buchberger} is nontrivial in general. Computing syzygy modules is nontrivial, so it should be assigned an appropriate cost. More refined information on the structure of the syzygy modules appearing may also be needed in order to give a useful bound. However, many bounds in the literature on Groebner bases only bound some partial aspect of Buchberger's algorithm, such as the degree or number of elements of the final Groebner basis and output, or the degree of all polynomials appearing in intermediate steps. Extending these to Macaulay bases and $H$-bases would be of interest and may not require using as much information on the computation or structure of the syzygy modules.

Doubly exponential lower bounds are known on the computation of syzygy modules \cite{Bayer_Stillman_1988}, which imply doubly exponential degree lower bounds for both Groebner bases and Macaulay $H$-bases.

\begin{example}
Consider the grading of $S^{\oplus n}$ by $(\mathbb{N}^{d})^{\oplus n}$ with an appropriate term order. If $m_{1}, \dots, m_{n}$ are monomials in the free module $S^{\oplus n}$, then it is well-known that $\Syz(m_{1}, \dots, m_{n})$ is generated by
\[ \left\{ e_{i}\frac{\operatorname{lcm}(m_{i}, m_{j})}{m_{i}}  - e_{j}\frac{\operatorname{lcm}(m_{i}, m_{j})}{m_{j}} \middle| 1 \leq i, j \leq n \right\} \]
which corresponds exactly to an $S$-pair in the usual version of Buchberger's algorithm for ideals and modules over a polynomial ring. However, using Buchberger's algorithm for other gradings requires the ability to compute generators of $\Syz(\lf m_{1}, \dots, \lf m_{n})$, which in general is nontrivial.
\end{example}

\subsection{Lifting of syzygies}
With the setup of the generalized Buchberger algorithm in place, we can state an equivalent definition of Macaulay bases in terms of lifting of syzygies, as with Groebner bases.

\begin{definition}
Let $X = \{m_{1}, \dots, m_{n}\}$ generate $M$. A \textbf{lift} of a syzygy $s \in \Syz(\lf m_{1}, \dots, \lf m_{n})$ is a syzygy $t \in \Syz(m_{1}, \dots, m_{n})$ such that $\lf t = s$ with respect to the grading of Lemma~\ref{lem:gradsyz}.
\end{definition}

\begin{proposition}
A generating set $X = \{m_{1}, \dots, m_{n}\}$ of $M$ is a Macaulay basis iff every homogeneous syzygy in $\Syz(\lf m_{1}, \dots, \lf m_{n})$ has a lift.
\end{proposition}
\begin{proof}
Suppose $X$ is a Macaulay basis. Let $s \in \Syz(\lf m_{1}, \dots, \lf m_{n})$ be homogeneous. Let $m = \sum _{i = 1} ^{n} s_{i}m_{i}$ and $b \in B$ such that $\deg(s_{i}m_{i}) = b$, since $s$ is homogeneous. Then $\deg m < \max _{i} \deg(s_{i}m_{i})$, since $m_{b} = \sum _{i = 1} ^{n} s_{i}\lf m_{i} = 0$. Since $X$ is a Macaulay basis, there are $r_{1}, \dots, r_{n}$ such that $\sum _{i = 1} ^{n} r_{i}m_{i} = m$ and $\deg(r_{i}m_{i}) \leq \deg m$ for all $i$. Thus $(s_{1} - r_{1}, \dots, s_{n} - r_{n})$ is a lift of $s$.

Suppose every homogeneous syzygy has a lift. Let $s \in \Syz(\lf m_{1}, \dots, \lf m_{n})$ be homogeneous and $t \in \Syz(\lf m_{1}, \dots, \lf m_{n})$ be a lift of $s$. Let $r_{i} = t_{i} - s_{i}$ for all $i$. Then $\deg r_{i} < \deg t_{i} = \deg s_{i}$, and
\[ \sum _{i = 1} ^{n} s_{i}m_{i} = \sum _{i = 1} ^{n} r_{i}m_{i}. \]
Thus $\sum _{i = 1} ^{n} s_{i}m_{i} \rightarrow^{*} 0$. This holds for an arbitrary syzygy, so $X$ is a Macaulay basis by Theorem~\ref{thm:conditions}.
\end{proof}

\subsection{Uniqueness}
A standard fact about Groebner bases and $H$-bases is that a reduced Groebner basis (i.e.\ one where every element is reduced with respect to all the others) for a given ideal is unique \citelist{\cite{cocoa1}*{Theorem 2.4.13} \cite{ms2003}*{Theorem 6.5}}. The following theorem generalizes this to Macaulay bases of modules.

\begin{theorem}[]
\label{thm:uniqueness}
Let $X = \{m_{1}, \dots, m_{n}\}$ and $X' = \{m'_{1}, \dots, m'_{n'}\}$ be reduced Macaulay bases of $M$, and set
\begin{align*}
X_{b} = \{m_{i} : 1 \leq i \leq n, \deg m_{i} = b\} \\
X'_{b} = \{m'_{i} : 1 \leq i \leq n', \deg m'_{i} = b\}
\end{align*}
for all $b \in B$. Then $|X_{b}| = |X'_{b}|$ for all $b \in B$ and $n = n'$. If $N$ is a complemented graded vector space and the Macaulay bases are reduced with respect to $\Rightarrow^{*}$, then for each $b$ there is an invertible matrix $Q_{b} \in \operatorname{GL}_{\mathbf{k}}(|X_{b}|)$ expressing the elements of $X'_{b}$ as a linear combination of the elements of $X_{b}$. If the complementation on $N$ is given by an inner product, then $Q_{b}Q_{b}^{*}$ and $Q_{b}^{*}Q_{b}$ are diagonal.
\end{theorem}
\begin{proof}
Let $L = \lf M$. We use induction on $b$ to show that the statement in the theorem holds, and also that $W_{b'}(Y) = W_{b'}(Y')$ where $Y = \bigcup _{b' \leq b} X_{b'}$ and $Y' = \bigcup _{b' \leq b} X'_{b'}$. Suppose we have shown the statement for all $b' < b$ and wish to show it for $b$. Then $L_{b}$ is equal to the $\mathbf{k}$-spans of both $W_{b}(Y) \cup X_{b}$ and $W_{b}(Y') \cup X'_{b}$. By the induction hypothesis, $W_{b}(Y) = W_{b}(Y')$. Since $X$ and $X'$ are both reduced, we must have 
\[ |X_{b}| = \dim_{\mathbf{k}} L_{b} - \dim_{\mathbf{k}} W_{b}(Y) = \dim_{\mathbf{k}} L_{b} - \dim_{\mathbf{k}} W_{b}(Y') = |X'_{b}|. \]
Now suppose that $N$ is complemented graded. Then the fact that $X$ and $X'$ are reduced implies that
\begin{align*}
\operatorname{span}_{\mathbf{k}} X_{b} &= \bigoplus _{b' \leq b} L_{b'} \ominus W_{b'}(Y) \\
\operatorname{span}_{\mathbf{k}} X'_{b} &= \bigoplus _{b' \leq b} L_{b'} \ominus W_{b'}(Y')
\end{align*}
and that $X_{b}$ and $X'_{b}$ are bases for the right-hand sides. By the induction hypothesis the right-hand sides are equal, so $X_{b}$ and $X'_{b}$ are both bases for the same vector space and thus related by an invertible matrix $Q_{b} \in \operatorname{GL}_{\mathbf{k}}(|X_{b}|)$. If the complementation is given by an inner product then $X_{b}$ and $X'_{b}$ are in fact orthogonal bases, so $Q_{b}^{*}Q_{b}$ and $Q_{b}Q_{b}^{*}$ are diagonal. If we additionally assume that every $x \in X$ has unit norm with respect to the inner product, then $Q_{b}^{*}Q_{b} = Q_{b}Q_{b}^{*} = I_{|X_{b}|}$ is an identity matrix.
\end{proof}

\section{Macaulay bases under a group action}
\label{sec:groupact}
In this section, we assume we have an algebraic group $G$ acting on $R$ homogeneously, so that $G \cdot R_{a} \subseteq R_{a}$. A common example is where $G$ acts on $\mathbb{C}^{d}$ and we extend to $\mathbb{C}[x_{1}, \dots, x_{d}]$ by $(g \cdot p)(x) = p(g^{-1} \cdot x)$. In this case we $G$ acts homogeneously with respect to the standard grading.

Many times one starts with such a $G$-action and $G$-invariant elements $p_{1}, \dots, p_{k} \in R$, and would like to do computations with the ideal $I = (p_{1}, \dots, p_{k})$. Since $I$ is also $G$-invariant, we may hope to compute a Groebner basis of $I$ in a way which preserves the $G$-symmetry. However, this almost always fails, as illustrated by the following example.

\begin{example} \label{ex:grobsym}
Consider $I = (x_{1}^{2} + x_{2}^{2} - 1, x_{1}^{2}x_{2}^{2} - 1)$ in $\mathbb{Q}[x_{1}, x_{2}]$ with $S_{2}$ permuting the variables. Computing a Groebner basis in Sage using the degrevlex monomial order gives the generating set $(x_{2}^{4} - x_{2}^{2} + 1, x_{1}^{2} + x_{2}^{2} - 1) = I$, which has broken the symmetry of the original generating set.
\end{example}

In contrast, there can be a ``nearly $G$-invariant'' Macaulay basis of a module $M$ when the action of $G$ respects the grading and inner product on $M$. We formalize this in the following theorem:

\begin{theorem}
\label{thm:ginv-basis}
Suppose $G$ acts homogeneously on $R$ and $N$. Suppose further that $N$ is a graded complemented vector space such that $W^{c}$ is $G$-invariant whenever $W$ is $G$-invariant. Then there is a reduced Macaulay basis $m_{1}, \dots, m_{n}$ of $M \leq N$ whose $\mathbf{k}$-span is a $G$-invariant vector space.
\end{theorem}
\begin{proof}
Let $L = \lf M$. Using induction on $b \in B$, we define a family of sets $\{X_{b}\}_{b \in B}$ such that $X_{b} \subseteq M$ for all $b$, $\deg x = b$ for all $x \in X_{b}$, $\operatorname{span}_{\mathbf{k}} X_{b}$ is $G$-invariant for all $b$, and $\bigoplus _{b' < b} L_{b'}$ is contained in the submodule generated by $\lf (\bigcup _{b' \leq b} X_{b'})$. For the base case $b = 0$, we can take $X_{b} = \emptyset$ since there is no $b' \in B$ such that $b < 0$.

For the induction step, suppose we have defined $X_{b'}$ for all $b' < b$ and need to define $X_{b}$. Let $Y = \bigcup _{b' < b} X_{b'}$ and consider the subspace
\[ V_{b} = \operatorname{span}_{\mathbf{k}} \{ r \lf m \mid r \in R, m \in Y, \deg rm = b \} \leq L_{b}. \]
We claim that $V_{b}$ is $G$-invariant. Suppose $m \in V$ with $m = r \lf x$ for some $x \in Y$ and $r \in R$. Then
\begin{align*}
g \cdot m &= (g \cdot r) (g \cdot \lf x) \\
&= (g \cdot r) (\lf g \cdot x) \\
&= (g \cdot r) \left(\lf \sum _{y \in X_{b'}} \lambda_{y} y\right) \\
&= (g \cdot r) \sum _{y \in X_{b'}} \lambda_{y} \lf y \\
&= \sum _{y \in X_{b'}} \lambda_{y} (g \cdot r) \lf y \\
&\in V_{b}
\end{align*}
where the second equality follows because $G$ acts homogeneously, the third follows because $\operatorname{span}_{\mathbf{k}} X_{b'}$ is $G$-invariant, the fourth follows because $\deg y = b'$ for all $y \in X_{b'}$, and the last follows because $G$ acts homogeneously. Since $V_{b}$ is spanned by elements of the same form as $m$, this shows that $V_{b}$ is $G$-invariant.

Let $W_{b} \leq L_{b}$ be a $G$-invariant complement of $V_{b}$. Let $W_{b} = \bigoplus _{i = 1} ^{l} W_{i}$ be a decomposition of $W_{b}$ into irreducible representations. For each $W_{i}$ choose an arbitrary $w_{i} \in W_{i}$. Let $w_{i}$ be written as $w_{i} = \sum _{j} ^{k_{i}} r_{ij} \lf m_{ij}$ for some $r_{ij} \in R$ and $m_{ij} \in M$ with $1 \leq j \leq k_{i}$. By splitting each $r_{j}$ into homogeneous components and canceling we may assume that $\deg r_{ij} m_{ij} = b$. Let $x_{i} \in M$ be the normal form of $\sum _{j = 1} ^{k_{i}} r_{ij}m_{ij}$ with respect to $Y$, which clearly satisfies $\lf x_{i} = w_{i}$ since $W_{b} \perp V_{b}$. For each $i$, there is a basis $X_{b,i} \subseteq \{g \cdot x_{i} : g \in G\}$ of $W_{i}$, since the second set spans $W_{i}$. Set $X_{b} = \bigcup _{i = 1} ^{l} X_{b,i}$.

For each $i$ we have $\operatorname{span}_{\mathbf{k}} \lf X_{b,i} = W_{i}$, so $L_{b} = V_{b} \oplus W_{b}$ is contained in the submodule generated by $\bigcup _{b' \leq b} X_{b}$, and by the induction hypothesis so is $\bigoplus _{b' \leq b} L_{b}$. By construction $\operatorname{span}_{\mathbf{k}} X_{b,i}$ is $G$-invariant since it is spanned the orbit of $x_{i}$ under $G$, so $\operatorname{span}_{\mathbf{k}} X_{b}$ is as well. Clearly $x_{i} \in M$ and so $g \cdot x_{i} \in M$ for all $g \in G$ and $1 \leq i \leq l$ since $M$ is $G$-invariant, so $X_{b} \subseteq M$. We have
\[ \lf g \cdot x_{i} = g \cdot \lf x_{i} = g \cdot w_{i} \in M_{b} \]
since $G$ acts homogeneously and $\lf x_{i} = w_{i}$, so $\deg x = b$ for all $x \in X_{b}$. To see that $g \cdot x_{i}$ is reduced with respect to $Y$ for any $x_{i}$ and $g \in G$, note that $(x_{i})_{b'} \in V_{b'}^{c}$ for any $b' < b$, so $(g \cdot x_{i})_{b'} = g \cdot (x_{i})_{b'} \perp V_{b'}^{c}$ since $V_{b'}$ and $V_{b'}^{c}$ are $G$-invariant. The Hilbert basis theorem then implies that $\bigcup _{b \in B} X_{b}$ is finite.
\end{proof}

\begin{remark}
In Theorem~\ref{thm:ginv-basis} the Macaulay basis we construct is easily seen to be reduced. The uniqueness result of Theorem \ref{thm:uniqueness} then implies that any reduced Macaulay basis of $M$ has a $G$-invariant $\mathbf{k}$-span as long as the inner product used is also $G$-invariant.
\end{remark}

Moreover, we can show that the normal form algorithm is $G$-equivariant:

\begin{proposition}
Assume the conditions of Theorem~\ref{thm:ginv-basis}. Let $X$ be a Macaulay basis of $M$, and
\[ p : N \to \bigoplus _{b \in B} N_{b} \ominus W_{b}(X) \]
be the map sending $m \in N$ to its normal form $m' \in N$ with respect to $X$. Then $p$ is $G$-equivariant.
\end{proposition}
\begin{proof}
By assumption, $W_{b}(X) \leq N_{b}$ is a $G$-invariant subspace. Since $p$ is the $\mathbf{k}$-linear projection with kernel $M$, which is also $G$-invariant, it follows that $p$ is just the quotient map and thus $G$-equivariant.
\end{proof}

\begin{example}
Consider the ideal $I = (x_{1}^{2} + x_{2}^{2} - 1, x_{1}^{2}x_{2}^{2} - 1)$ from Example~\ref{ex:grobsym} where $\mathbb{Q}[x_{1}, x_{2}]$ has the standard grading and $S_{2}$ again permutes the variables. By computing the syzygy module using a standard Groebner basis package and manually reducing the resulting polynomials, we can check that the original set of generators is in fact a Macaulay $H$-basis and its $\mathbf{k}$-span is $S_{2}$-invariant.
\end{example}

\begin{example}
Consider $C_{4} = \langle c \rangle$ acting on $\mathbb{Q}[x_{1}, x_{2}]$ by $c \cdot x_{1} = -x_{2}$ and $c \cdot x_{2} = x_{1}$. Let
\[
I = (x_{1}^{2} + x_{2}^{2} - 1, x_{1}^{2}x_{2}^{2}, x_{1}^{3}x_{2} - x_{1}x_{2}^{3}).
\]
Then we can then check that a Macaulay $H$-basis is given by
\[
I = (x_{1}^{2} + x_{2}^{2} - 1, x_{1}^{2}x_{2}^{2}, x_{1}^{3}x_{2} - x_{1}x_{2}^{3}, x_{1}x_{2}^{2}, x_{1}^{2}x_{2}, x_{1}x_{2}).
\]
Taking the $\mathbb{Q}$-span of the Macaulay $H$-basis, we see that it has three trivial subrepresentations spanned by $x_{1}^{2} + x_{2}^{2} - 1$, $x_{1}^{2}x_{2}^{2}$, and $x_{1}^{3}x_{2} - x_{1}x_{2}^{3}$; one copy of the defining representation (which is irreducible over $\mathbb{Q}$) spanned by $x_{1}x_{2}^{2}$ and $x_{1}^{2}x_{2}$; and a one-dimensional representation spanned by $x_{1}x_{2}$, which is isomorphic to the determinant of the defining representation.
\end{example}

\section{Applications}
\label{sec:applications}
\subsection{Elimination}
\label{sec:elimination}
A common application of Groebner bases is \textit{elimination}, which is the task of computing $M \cap \hat{N}$ where $\hat{N}$ is an $\hat{R}$-submodule of $N$ for some subring $\hat{R} \leq R$. Algorithms for this problem are known for the case $R$ is a polynomial ring and $\hat{R}$ is the polynomial subring in a subset of the variables. These algorithms compute a Groebner basis with respect to an appropriate type of monomial order called an \textit{elimination order}, for which the elements of the Groebner basis also contained in $\hat{N}$ will be a Groebner basis of $\hat{N}$. We generalize this to gradings satisfying an appropriate elimination condition and prove the analogous result.

\begin{definition}
\label{def:elimgrading}
Suppose $\hat{N} \leq N$ is an $\hat{R}$-submodule for some subring $\hat{R} \leq R$. Suppose further that $\hat{R}$ and $\hat{N}$ are respectively graded by totally ordered commutative submonoids $\hat{A} \leq A$ and $\hat{B} \leq B$ such that $\hat{R}_{a} = R_{a}$ and $\hat{N}_{b} = N_{b}$ for $a \in \hat{A}$ and $b \in \hat{B}$, and that $\hat{A}\hat{B} \subseteq \hat{B}$.
The grading of $N$ by $B$ is an \textbf{elimination grading} if $\deg_{B} n \in \hat{B}$ implies $n \in \hat{N}$ for all $n \in N$, and $\deg_{A} r \in \hat{A}$ implies $r \in \hat{A}$ for all $r \in R$.
\end{definition}

\begin{example}
Given an elimination order on a polynomial ring, the grading defined in Example~\ref{ex:mongrad} is also an elimination grading.
\end{example}

\begin{example}
Consider $\hat{S} = \mathbf{k}[x_{1}, \dots, x_{k}] \leq S = \mathbf{k}[x_{1}, \dots, x_{n}]$. We grade $S$ by $\mathbb{N}^{2}$ where $\deg x_{i} = (1, 0)$ if $1 \leq i \leq k$ and $\deg x_{i} = (0, 1)$ if $k + 1 \leq i \leq n$, and where the ordering is one coming from an elimination ordering on two variables. Then $\hat{A} = \mathbb{N} \times \{0\}$.
From our choice of an elimination ordering, we see that if $\deg p \in \hat{A}$ then all homogeneous components of $p$ are in $\hat{A}$, so $p \in \hat{S}$. Thus this defines an elimination grading.
\end{example}

The following proposition generalizes Theorem~3.4.5 of \cite{cocoa1} on eliminating variables using Groebner bases.

\begin{proposition}
Suppose we have the given setup of Definition~\ref{def:elimgrading} and a submodule $M \leq N$. Then
\begin{enumerate}
\item $\lf_{\hat{B}}(M \cap \hat{N}) = \lf_{B}(M) \cap \hat{N}$
\item if $X$ is an $B$-Macaulay basis of $M$, then $X \cap \hat{N}$ is an $\hat{B}$-Macaulay basis of $M \cap \hat{N}$.
\end{enumerate}
\end{proposition}
\begin{proof}
If $m \in M \cap \hat{N}$, then $\lf_{\hat{B}} m = \lf_{B} m$ is in both $\lf_{B} M$ and $\hat{N}$, so $\lf_{\hat{B}}(M \cap \hat{N}) \subseteq \lf_{B}(M) \cap \hat{N}$. For the other inclusion, let $X = \{m_{1}, \dots, m_{n}\}$ be a $B$-Macaulay basis of $M$.
If $m \in \lf_{B}(M)$ and $m \in \hat{N}$, then we have a representation $m = \sum _{i = 1} ^{n} r_{i}\lf(m_{i})$ where $\deg_{B} r_{i}m_{i} \leq \deg_{B} m$. By splitting each $r_{i}$ into homogeneous components and equating homogeneous components on both sides, we may assume each $r_{i}$ is homogeneous. Thus $\deg _{B} r_{i}\lf(m_{i}) = \deg _{A}(r_{i})\deg_{B}(m_{i}) \in \hat{B}$. It then suffices to show that $\deg_{A} r_{i} \in \hat{A}$ and $\deg _{B} m_{i} \in \hat{B}$, as then $r_{i}m_{i} \in \hat{B}$.

To that end, we need to show that if $a \in A$ and $b \in B$ such that $ab \in \hat{B}$, then $a \in \hat{A}$ and $b \in \hat{B}$. Since the action of $A$ is monotone, we have $b = 0 \cdot b \leq a \cdot b \in \hat{B}$, so $b \in \hat{B}$ by the elimination property.
To show that $a \in \hat{A}$, first note that by the elimination property $\hat{A}$ and $\hat{B}$ are intervals in $A$ and $B$.
If $0 < a_{1} < a_{2} < \cdots$ is a chain containing all elements of $\hat{A}$, then $0 < a_{1} \cdot b < \cdots$ is a countable strictly increasing chain in $\hat{B}$, so for every $x \in \hat{B}$ there is some $\hat{a} \in \hat{A}$ such that $x \leq \hat{a} \cdot b$. Suppose that $a \notin \hat{A}$. Then $a \geq \hat{a}$ and $a \cdot b \geq \hat{a} \cdot b$ for all $\hat{a} \in \hat{A}$. Since $a \cdot b \in \hat{B}$ there is a $\hat{a} \in \hat{A}$ such that $a \cdot b \leq \hat{a} \cdot b$, which contradicts the assumption that $a \notin \hat{A}$. Hence $a \in \hat{A}$.
\end{proof}

\subsection{Computing syzygy modules} Another standard application of Groebner bases is Schreyer's theorem, which gives a strategy for computing syzygy modules. Specifically, for any Groebner basis of a module there is an associated Groebner basis of the syzygy module of the original Groebner basis, where the syzygy module has a specific term order \cite{cocoa1}*{Proposition 3.1.4}. The next proposition gives a generalization to Macaulay bases, using refinements of the grading defined in Lemma~\ref{lem:gradsyz}.

\begin{proposition} \label{prop:compsyz}
Suppose $m_{1}, \dots, m_{n}$ is a Macaulay basis of $M$ and $s_{1}, \dots, s_{k}$ is a homogeneous generating set of $\Syz(\lf m_{1}, \dots, \lf m_{n})$. Let $s_{j}' \in \Syz(m_{1}, \dots, m_{n})$ be the syzygy obtained by reducing $\sum _{i = 1} ^{n} (s_{j})_{i}m_{i}$ to 0. If $s_{1}, \dots, s_{j}$ is also a Macaulay basis of $\Syz(\lf m_{1}, \dots, \lf m_{n})$ with respect to a grading of $R^{\oplus n}$ by $B'$ which refines the grading by $B$ defined in Lemma~\ref{lem:gradsyz}, then $s_{1}', \dots, s_{j}'$ is a $B'$-Macaulay basis of $\Syz(m_{1}, \dots, m_{n})$.
\end{proposition}
\begin{proof}
It is clear that $\lf_{B} s_{j}' = s_{j}$, since when reducing $\sum _{i = 1} ^{n} (s_{j})_{i}m_{i}$ to 0 all the additional terms must have degree less than $\deg_{B} s_{j}$. Since the grading by $B'$ refines the grading by $B$, we have $\lf_{B'} s_{j}' = \lf _{B'} s_{j}$. If $s \in \Syz(m_{1}, \dots, m_{n})$ then equating homogeneous components on both sides of $\sum _{i = 1} ^{n} s_{i}m_{i} = 0$ shows that $\lf_{B} s \in \Syz(\lf m_{1}, \dots, \lf m_{n})$. Thus
\begin{align*}
\lf_{B'} s &= \lf_{B'} (\lf_{B} s) \\
&\in (\lf_{B'} s_{1}, \dots, \lf_{B'} s_{k}) \\
&= (\lf_{B'} s_{1}', \dots, \lf _{B'}s_{k}')
\end{align*}
where the first equality is because the grading by $B'$ refines that of $B$, the second is because $s_{1}, \dots, s_{k}$ is a $B'$-Macaulay basis of $\Syz(\lf m_{1}, \dots, \lf m_{n})$, and the last follows from the argument in the previous paragraph. Thus by Theorem~\ref{thm:conditions}, $s_{1}', \dots, s_{j}'$ is a $B'$-Macaulay basis of $\Syz(m_{1}, \dots, m_{n})$.
\end{proof}

\begin{example}
Any grading refines itself, so we may take both gradings in Proposition~\ref{prop:compsyz} to be the grading defined in Lemma~\ref{lem:gradsyz}.
\end{example}

\begin{example}
Schreyer's theorem is usually stated as giving a Groebner basis of the syzygy module with respect to a term order, which is maximal refinement of the grading in Lemma~\ref{lem:gradsyz} where each graded component of the refined grading has dimension 1 over $\mathbf{k}$. When working only with Groebner bases, this is necessary because the grading defined in Lemma~\ref{lem:gradsyz} may have graded components with dimension greater than 1. In order to get a term order and Groebner basis we must maximally refine the grading, typically by breaking ties lexicographically among the $m_{i}$. The generalization in Proposition~\ref{prop:compsyz} makes it clear how the different gradings interact with each other.
\end{example}

\subsection{Macaulay bases of graded submodules and refined gradings}
In many situations, such as when dealing with projective varieties, a given submodule is graded and Groebner bases satisfy some additional properties. In particular, for the standard grading one can reduce computation of the Hilbert series of an ideal $I$ to computing the Hilbert series of the monomial ideal $\lf I$, which can then be reduced to a combinatorial problem. We provide generalizations of these facts here.

Throughout this subsection, we assume that $N$ has a grading by $B$ which is refined by a grading by $B'$. When $B$ is the standard grading on a polynomial ring and $B'$ is a monomial order we recover the case described before.

\begin{proposition} \label{prop:homogbasis}
Suppose that $M$ is a $B$-graded submodule of $N$. Then any reduced Macaulay basis of $M$ with respect to $B'$ consists of only $B$-homogeneous elements.
\end{proposition}
\begin{proof}
Let $X = \{m_{1}, \dots, m_{n}\}$ be a reduced $B'$-Macaulay basis of $M$. Since $M$ is a $B$-graded submodule, and $B$-homogeneous component of any $m \in M$ is also contained in $M$, and in particular $\lf_{B} m_{i} \in M$ for any $i$. Thus $m_{i} - \lf_{B} m_{i} \in M$ and reduces to 0 with respect to $X$. However, since the given Macaulay basis is reduced we must have $m_{i} - \lf_{B} m_{i} = 0$, as otherwise we would be able to apply a nontrivial reduction to $m_{i}$. Thus $m_{i} = \lf_{B} m_{i}$ for every $m_{i}$ and is thus $B$-homogeneous.
\end{proof}

\begin{proposition}
Let $M$ be a $B$-graded submodule. Then $\lf_{B'} M$ is a $B$-graded submodule and $\dim M_{b} = \dim (\lf_{B'} M)_{b}$ for all $b \in B$.
\end{proposition}
\begin{proof}
Let $X = \{m_{1}, \dots, m_{n}\}$ be a Macaulay basis of $M$ with respect to $B'$. Then $\lf_{B'} M$ is generated by $\lf_{B'} m_{1}, \dots, \lf_{B'} m_{n}$. By Proposition~\ref{prop:homogbasis}, each $m_{i}$ is $B$-homogeneous, so each $\lf_{B'} m_{i}$ is as well and $\lf_{B'} M$ is a $B$-graded submodule.

Choose an arbitrary complementation on $N$ which respects the $B'$-grading. Let $f : B' \to B$ be the monotone map as in the definition of refinement and
\[ Y = \{\lf_{B'} m_{i} : 1 \leq i \leq n\}. \]
Then $W_{b'}(X) = W_{b'}(Y)$ for all $b' \in B'$, and $Y$ is a Macaulay basis of $\lf_{B'} M$ with respect to the $B'$-grading since $\lf_{B'} M$ is a $B'$-graded submodule. By Corollary~\ref{cor:normform} we have
\begin{align*}
N_{b} &= M_{b} \oplus \left( \bigoplus _{b' \in f^{-1}(\{b\})} N_{b'} \ominus W_{b'}(X) \right) \\
&= \bigoplus _{b' \in f^{-1}(\{b\})} (\lf_{B'} M)_{b'} \oplus \left( \bigoplus _{b' \in f^{-1}(\{b\})} N_{b'} \ominus W_{b'}(Y) \right) \\
&= (\lf_{B'} M)_{b} \oplus \left( \bigoplus _{b' \in f^{-1}(\{b\})} N_{b'} \ominus W_{b'}(X) \right).
\end{align*}
Thus
\[ \dim M_{b} = \dim (\lf_{B'} M)_{b} = \dim \left( \bigoplus _{b' \in f^{-1}(\{b\})} N_{b'} \ominus W_{b'}(X) \right). \]
\end{proof}

\begin{corollary}
If $H(M, b)$ denotes the Hilbert function of the graded submodule $M$, then $H(M, b) = H(\lf_{B'} M, b)$.
\end{corollary}

Given a group $G$ acting on $N$ and an invariant graded submodule $M$, one could also ask for the decomposition of $M_{b}$ into irreducible representations. A similar argument as in the above proposition shows that $M_{b}$ and $(\lf_{B'} M)_{b}$ are isomorphic as $G$-representations. Of course, the simplest case is when $B'$ is a monomial order, but ideally this would be avoided since it would likely destroy the symmetry. It is an interesting problem to given an efficient algorithm for computing the Molien series of $G$ in this case.

\subsection{Macaulay \texorpdfstring{$H$}{H}-bases and homogenization} 
\label{sec:hbasis}
In this section we assume that $R$ is a polynomial ring with the standard grading and consider Macaulay $H$-bases of modules over $R$ graded by $\mathbb{N}$. The standard grading has the special property that we can homogenize and dehomogenize polynomials with respect to it, and these operations are inverse to each other up to a power of the homogenizing variable. We show how Macaulay $H$-bases can be used to compute homogenizations of modules, including the projective closure of the ideal of an affine variety. Recall that any Groebner basis with respect to a degree-graded monomial order is a Macaulay $H$-basis, so our results include ones for degree-graded Groebner bases as a special case.

Let $t$ be a distinguished variable of $R$ and let $T = \mathbf{k}[x_{1}, \dots, x_{d}]$ be the polynomial ring on all generators $x_{1}, \dots, x_{d}$ of $R$ except for $t$. In the rest of this section we will assume $N = \bigoplus _{i = 1} ^{n} R(-a_{i})$ with the standard $\mathbb{N}$-grading, where $R(-a_{i})$ is the graded $R$-module with underlying module $R$ and grading shifted by $a_{i}$.

\begin{definition}
Let $m \in N$. The \textbf{dehomogenization} of $m = (m_{1}, \dots, m_{n}) \in N$ is \[m^{D} = (m_{1}|_{t = 1}, \dots, m_{n}|_{t = 1}) \in T^{\oplus n}.\] The \textbf{homogenization} of $m = (m_{1}, \dots, m_{n}) \in T^{\oplus n}$ with respect to $N$ is
\[ m^{H} = t^{\max_{j} a_{j} + \deg m}(t^{-a_{1}}m_{1}(x_{1}/t, \dots, x_{d}/t), \dots, t^{-a_{n}}m_{n}(x_{1}/t, \dots, x_{d}/t)). \]
\end{definition}

Generalizing the known results for $H$-bases of ideals \cite{cocoa2}*{Proposition 4.3.19}, we can describe how Macaulay $H$-bases of modules interact with homogenization.

\begin{proposition}
\label{prop:comphom}
The following are equivalent:
\begin{enumerate}
\item The module $M^{H} \leq N$ is generated by $m_{1}^{H}, \dots, m_{n}^{H}$.
\item The elements $m_{1}, \dots, m_{n}$ are a Macaulay $H$-basis of $M \leq T^{\oplus n}$.
\end{enumerate}
\end{proposition}
\begin{proof}
(1) $\Rightarrow$ (2) Suppose $M^{H}$ is generated by $m_{1}^{H}, \dots, m_{n}^{H}$. If $m \in M$ then $m^{H} \in M^{H}$, so $m^{H} = \sum _{i = 1} ^{n} r_{i}m_{i}^{H}$. By splitting each $r_{i}$ into its homogeneous components we may assume each $r_{i}$ is homogeneous. Then
\begin{align*}
\deg r_{i} &= \deg m^{H} - \deg m_{i}^{H} \\
&= \deg m - \deg m_{i}
\shortintertext{and thus}
m &= (m^{H})^{D} \\
&= \left(\sum _{i = 1} ^{n} r_{i}m_{i}^{H} \right)^{D} \\
&= \sum _{i = 1} ^{n} r_{i}^{D}(m_{i}^{H})^{D} \\
&= \sum _{i = 1} ^{n} r_{i}^{D}m_{i}.
\end{align*}
We see $\deg r_{i}^{D} \leq \deg r_{i} = \deg m - \deg m_{i}$, so $\deg r_{i}^{D}m_{i} \leq \deg m$. By Theorem \ref{thm:conditions}, $m_{1}, \dots, m_{n}$ is a Macaulay $H$-basis of $M$.

(2) $\Rightarrow$ (1) Suppose $m_{1}, \dots, m_{n}$ is a Macaulay $H$-basis of $M$. By Theorem \ref{thm:conditions} $m \in M$ can be written as $m = \sum _{i = 1} ^{n} r_{i}m_{i}$ where $\deg r_{i}m_{i} \leq \deg m$. From the definition of homogenization, we see
\begin{align*}
m^{H} &= \sum _{i = 1} ^{n} t^{\max_{j} a_{j} + \deg m - \deg r_{i}m_{i}} (r_{i}m_{i})^{H} \\
&= \sum _{i = 1} ^{n} t^{\max_{j} a_{j} + \deg m - \deg r_{i}m_{i}} r_{i}^{H}m_{i}^{H}
\end{align*}
where the first equality comes from the fact that $\deg r_{i}m_{i} \leq \deg m$ and the definition of homogenization.
\end{proof}

\section{Conclusion and further work}
\subsection{Algorithmic issues and computing Macaulay bases} \label{sec:issues}
So far, we have seen how Macaulay bases generalize many useful properties of Groebner and Macaulay $H$-bases. This includes a generalization of Buchberger's algorithm, but as with $H$-bases this requires computing generators of the syzygy module of leading forms. There has been some work on doing this for $H$-bases by Gaussian elimination on extremely large Macaulay matrices \cites{HJ2018, JS2018, yilmaz}. However, it is not clear that this is more efficient than computing a Groebner basis with respect to a degree-graded term order and then reducing it using the Macaulay $H$-basis reduction algorithm. It is stated in \cite{JS2018} that this process is numerically stable in floating point arithmetic; however, the numerical stability is not proven and the matrices involved could potentially be be very large compared to the input, which suggests that the worst-case numerical stability could be poor when the matrices are excessively large. It would be interesting to see if these issues can be circumvented to create an algorithm for computing Macaulay bases (or even just $H$-bases) more efficiently than computing a degree-graded Groebner basis. There is potential that an alternative Groebner basis algorithm, such as F4 \cite{faugere1999}, may be more suited for generalization to Macaulay bases. Additionally, it would be interesting to see if algorithms such as the Groebner walk or FGLM can be adapted to convert between Macaulay bases graded by the same monoid but with different total orderings.

It may be easier to find theoretical applications of Macaulay bases, especially when used as a relaxed replacement for Groebner bases. In ongoing work the author shows that an ideal appearing the sum-of-squares relaxation for an optimization problem in quantum information generically has a simple Macaulay basis, improving the bound on the running time of this algorithm compared to the bound in \cite{hnw2017}.

\subsection{Macaulay bases of noncommutative rings and group actions} We saw in Section~\ref{sec:applications} that the $\mathbf{k}$-span of a Macaulay basis of $M$ is $G$-invariant when $G$ acts on $N$ homogeneously, and that the output of the normal form algorithm is $G$-invariant. This suggests that many of the results on Macaulay bases can be generalized to Noetherian modules over suitable noncommutative rings, and in particular the \textit{crossed product} algebra $R \rtimes \mathbf{k}[G]$ of a ring $R$ acted upon by a group $G$, where the multiplication is defined by
\[
(r_{1} \otimes g_{1}) \cdot (r_{2} \otimes g_{2}) = r_{1}(g_{1} \cdot r_{2}) \otimes g_{1}g_{2}.
\]
One can verify that a $G$-invariant submodule of $R^{\oplus n}$ is the same as an $(R \rtimes \mathbf{k}[G])$-submodule of $R^{\oplus n}$ and is also Noetherian as an $(R \rtimes \mathbf{k}[G])$-module. Note that Groebner basis techniques cannot directly apply to the crossed product algebra: a finite group must have torsion and thus cannot have a compatible total order, so we cannot grade the crossed product in a way where each graded component has dimension 1 over $\mathbf{k}$. However, we can use the grading $(R \rtimes \mathbf{k}[G]))_{a} = R_{a} \otimes \mathbf{k}[G]$ when $G$ acts homogeneously, which makes it possible to apply Groebner basis techniques to the crossed product algebra in the form of Macaulay bases with respect to this grading. It would be very interesting to see how much of the standard theory can be carried even further to this or to other noncommutative settings.

\section{Acknowledgements}
The author would like to thank Alex Townsend and Pablo Parrilo for helpful discussions, and Anna Brosowsky and Connor Simpson for useful comments on earlier drafts of this paper. 

\bibliography{bib}

\end{document}